\documentclass[11pt]{amsart}
\usepackage{amsmath,amsfonts,amssymb,amsthm,amscd,latexsym,euscript,verbatim, bbold}
\usepackage[all]{xy}

\makeatletter
\def\section{\@startsection{section}{1}%
  \z@{1.1\linespacing\@plus\linespacing}{.8\linespacing}%
  {\normalfont\Large\scshape\centering}}
\makeatother

\theoremstyle{plain}

\newtheorem*{conj*}{Root Groups Conjecture}
\newtheorem*{thm1.2}{(1.2) Theorem}
\newtheorem*{thm1.3}{(1.3) Theorem}
\newtheorem*{thm1.4}{(1.4) Theorem}
\newtheorem*{prop*}{Proposition}

\newtheorem{prop}{Proposition}[section]

\newtheorem{thm}[prop]{Theorem}
\newtheorem{cor}[prop]{Corollary}
\newtheorem{lemma}[prop]{Lemma}

\theoremstyle{definition}

\newtheorem*{Def*}{Definition}

\newtheorem*{notation*}{Notation}
\newtheorem{remark}[prop]{Remark}

\newcommand{\cala}{\mathcal{A}}
\newcommand{\calb}{\mathcal{B}}

\newcommand{\ff}{\mathbb{F}}

\newcommand{\frakA}{\mathfrak{A}}

\newcommand{\ga}{\alpha}

\newcommand{\gd}{\delta}
\newcommand{\gD}{\Delta}

\newcommand{\gl}{\lambda}

\newcommand{\gvp}{\varphi}

\newcommand{\gs}{\sigma}

\newcommand{\gt}{\tau}

\newcommand{\sminus}{\smallsetminus}
\newcommand{\lan}{\langle}
\newcommand{\ran}{\rangle}

\newcommand{\Aut}{{\rm Aut}}

\newcommand{\half}{\textstyle{\frac{1}{2}}}

\newcommand{\e}{\mathbb{1}}

\numberwithin{equation}{section}

\begin{document}
\title[]{On primitive axial algebras of Jordan type} 
\author[J.~I.~Hall,\quad  Y.~Segev,\quad S.~Shpectorov]{J.I.~Hall\qquad  Y.~Segev\qquad S.~Shpectorov }
\address{Jonathan, I.~Hall\\
        Department of Mathematics\\
				Michigan State University\\
				Wells Hall, 619 Red Cedar Road, East Lansing, MI 48840\\
				United States}
\email{jhall@math.msu.edu}

\address{Yoav Segev \\
         Department of Mathematics \\
         Ben-Gurion University \\
         Beer-Sheva 84105 \\
         Israel}
\email{yoavs@math.bgu.ac.il}
\address{Sergey Shpectorov\\
         School of Mathematics\\
				 University of Birmingham\\
				Watson Building, Edgbaston, Birmingham, B15 2TT\\
				United Kingdom}
\email{s.shpectorov@bham.ac.uk}


\begin{abstract}
In this note we give an overview of our knowledge
regarding primitive axial algebras of Jordan type half and
connections between $3$-transposition groups and Matsuo algebras.
We also show that primitive axial algebras of Jordan type $\eta$
admit a Frobenius form, for any $\eta$.
\end{abstract}

 \dedicatory{Dedicated to Professor Robert L.~Griess, Jr.~on the occasion
     of his $71$st birthday}

\date{May 9, 2017}
\maketitle
\section{Introduction}

The purpose of this note is threefold.  In \S 2 we give an overview of our knowledge
regarding {\it primitive axial algebras of Jordan type half}. This is taken
from \cite{hss}.  In fact we focus in \S 2
on one of the main results in \cite{hss} which characterizes
Jordan algebras of Clifford type amongst primitive axial algebras of Jordan type half. 
The primitive axial algebras of Jordan type $\eta\ne\half$ are reviewed (amongst other things) by Jon Hall
in another paper of this volume.  In \S 3, we complete, for the case $\eta=\half,$ a result
connecting $3$-transposition groups  
and Matsuo algebras, established in \cite[Theorem 6.3]{hrs} for $\eta\ne\half$. 
In \S 4 we show that any primitive axial algebra of Jordan type $\eta$ (any $\eta$) 
admits a Frobenius form.

We start by recalling a few definitions.  We do not give the historical background 
as it can be best found in the introduction to \cite{hrs}.

All algebras $A$ in this note are {\bf commutative, non-associative} over
a field $\ff$ of {\bf characteristic not $2$}.  

For $a\in A$ the {\bf adjoint operator} ${\rm ad}_a$
is {\bf multiplication by $a$}, so
\[
{\rm ad}_a\colon A\to A,\ x\mapsto xa.
\]
An {\bf axis} in $A$ is, by definition, a {\bf semisimple idempotent}, i.e., an idempotent
whose minimal ad-polynomial 
has few distinct linear factors; where the minimal ad-polynomial is
the minimal polynomial of the linear operator ${\rm ad}_a$ (we are {\bf not} assuming
that $A$ is finite dimensional,  however, we are assuming that ${\rm ad}_a$ has 
a minimal polynomial). 
  
{\bf Axial algebras},
introduced recently by Hall, Rehren and Shpectorov ([HRS]),  
are, by definition, algebras   generated by axes.
When certain
{\bf fusion rules}, i.e.~multiplication rules,
between the eigenspaces corresponding to an axis,
are imposed the structure of axial 
algebras remains interesting  yet it is more rigid.  

Given  an element  $a\in A$ and a scalar $\lambda \in \ff,$ 
the $\gl$-eigenspace of ${\rm ad}_a$ is denoted $A_{\gl}(a),$ so:
\[
A_\lambda(a) := \{x \in A\ |\ xa = \lambda x\}\,.
\]
(We allow $A_\lambda(a)=0$.)
\medskip

\noindent
{\bf Axial algebras of Jordan type $\eta,$ where $\eta\notin\{0,1\}$ is fixed}, are 
algebras generated by a set of axes $\cala$ such that for each $a\in\cala:$
\begin{enumerate}
\item
The minimal ad-polynomial of $a$ divides $(x-1)x(x -\eta)$. 
 \item
The  {\bf fusion rules} imitate the {\bf Peirce multiplication rules} in Jordan algebras.
These fusion rules are:
\[
A_1 (a)A_1 (a) \subseteq A_1 (a)\quad\text{and}\quad A_0(a)A_0(a)\subseteq A_0(a),
\]
\[
A_1(a)A_0(a)=\{0\},
\]
\[
(A_0(a)+A_1(a))A_{\eta}(a)\subseteq A_{\eta}(a),\quad\text{and}\quad A_{\eta}(a)^2\subseteq A_0(a)+A_1(a).
\]
In particular, if we set
\[
A_+ (a) = A_1 (a) \oplus A_0(a)
\quad\text{and}\quad 
A_- (a) = A_\eta (a).
\]
then
\[
A_\delta (a)A_\epsilon (a) \subseteq A_{\delta\epsilon} (a)\,,
\]
for $\delta, \epsilon \in \{+,-\}.$
\end{enumerate}
Thus, for example, Jordan algebras are axial algebras of Jordan type $\half,$ provided
that they are generated by idempotents.

An axis $a \in A$ is {\bf absolutely primitive} if $A_1(a) = \ff a$
(this is stronger than the usual notion of primitivity). 
We call an absolutely primitive axis $a$ satisfying (1), (2) above  an {\bf $\eta$-axis}.

A {\bf primitive axial algebra of Jordan type $\eta$} is 
an algebra generated by $\eta$-axes. 
For $\eta \neq \half,$ primitive axial algebras of Jordan type $\eta$ were thoroughly analyzed by Hall, Rehren, and
Shpectorov in \cite{hrs}.
The case $\eta=\half$, is much less
understood and is of a different nature. This case is the focus
of \cite{hss} and of \S\S 2,3 of this note.

Given an $\eta$-axis  $a\in A,$ recall that
\[
A=\overbrace{A_1(a)\oplus A_0(a)}^{A_+(a)}\oplus \overbrace{A_{\eta}(a)}^{A_{-}(a)}.
\]
The map $\gt(a)\colon A\to A$
defined by $x^{\gt(a)}=x_{+}-x_{-},$ where $x=x_+ + x_{-}\in A_+(a)+A_{-}(a),$
is an automorphism of $A$ of order $1$ or $2$.
It is called {\bf the Miyamoto involution corresponding to $a$}.
\medskip

\subsection{Jordan algebras of Clifford type}\hfill
\medskip

A Jordan algebra of Clifford type $J(V,B)$ consists of the following information:
\begin{enumerate}
\item
A vector space $V$ over $\ff$ together with a symmetric bilinear form
$B$ on $V$.  The corresponding quadratic form is denoted $q(v)=B(v,v)$.

\item
The Jordan algebra $J(V,B)$ is $\ff\e\oplus V$ with multiplication
defined by 
\[
\e\text{ is the identity and}\ v\ast w=B(v,w)\e,\quad\forall v,w\in V.
\]
\end{enumerate}
 
\noindent
The algebra $J(V,B)$ comes from the associative {\bf Clifford algebra ${\rm Cl}(V,q)$}:
it is a sub-Jordan algebra of ${\rm Cl}(V,q)^+,$ where, as usual, $\frakA^+$ denotes
the special Jordan algebra that emerges from the associative algebra $\frakA.$  

Let $J=J(V,B)$. It is easy to check that:

\begin{itemize}
\item[(a)]
For $u\in V$ and $\ga\in \ff$, the element $\ga \e+u$ is an idempotent 
if and only if $\ga=\half$ and $q(u)=\frac{1}{4}$.

\item[(b)]
Assume that $a=\half \e+u$ is an idempotent in $J$.  Then
\begin{itemize}
\item[(i)]
 $J_1(a)=\ff a,$ so $a$ is a $\half$-axis.
(Thus $J(V,B)$ is a primitive axial algebra of Jordan type $\half$ iff
it is generated by idempotents.)

\item[(ii)]
$J_0(a)=\ff(\e-a)$  (of course $\e-a$ is a $\half$-axis), and

\item[(iii)]
$J_{\half}(a) =u^\perp=J_{\half}(\e-a),$ 
where $u^\perp=\{v\in V\mid B(u,v)=0\}$.
\end{itemize}
\item[(c)]
It follows that
$\gt(a)=\gt(\e-a),$ for any $\half$-axis $a$.
\end{itemize}
The purpose of \S2 is to show that property (c) above
essentially characterizes Jordan algebras of Clifford type
amongst primitive axial algebras of Jordan type $\half$.

\section{Primitive axial algebras of Jordan type half}

Throughout this section $A$ is a primitive axial algebra of Jordan type $\eta,$
generated by a set $\cala$ of $\eta$-axes.

Let $\gD$ be the graph
on the set of all  $\eta$-axes of $A,$ where distinct $a, b$ form
an edge iff $ab\ne 0$.  Let also $\gD_{\cala}$ be the full subgraph of $\gD$
on the set $\cala$.
The purpose of this section is to sketch a proof of the following theorem: 

\begin{thm}\label{thm main}
Assume that $\gD_{\cala}$ is connected and that there are two distinct
$\eta$-axes $a, b\in A$ such that $\gt(a)=\gt(b)$.
Then $\eta=\half,\ a+b=\e$ is the identity of $A,$   
and $A$ is a Jordan algebra of Clifford type.
\end{thm}  
 
In the remainder of this section we will sketch a
proof of Theorem \ref{thm main}.  
First we need  a theorem
that enables us to identify $A$
as a Jordan algebra of Clifford type in the case $\eta=\half$.

\begin{thm}\label{thm jvb}
Let $\eta=\half$. Assume that
$A$ contains two  $\half$-axes $a, b\in\cala$
such that $a+b=\e_A$   and
such that $v_av_c\in\ff\e_A,$ for all $c\in\cala,$ where $v_c=c-\half\e_A$.
Then  $A$ is a Jordan algebra of Clifford type.
\end{thm}

\noindent
We do not include a proof of Theorem \ref{thm jvb}, see \cite[Theorem 5.4]{hss}.

We will need some information about $2$-generated subalgebras of $A$.
This information is taken from \cite{hrs}.  Let $a,b\in\gD$ with $a\ne b$. We denote by   
{\bf $N_{a,b}$ the subalgebra generated by $a$ and $b.$
If $N_{a,b}$ contains an identity element, we denote it by $1_{a,b}$.}
Note that by \cite{hrs}, {\bf $2$-generated subalgebras
are at most $3$-dimensional}.

\begin{lemma}[Lemma 3.1.2 in \cite{hss}]\label{lem 2dim}
Let $a,b\in\gD$ with $a\ne b$.  Then $N_{a,b}$
is $2$-dimensional precisely in the following cases: 
 
\begin{enumerate}
\item
$ab=0;$ we then denote: $N_{a,b}=2B_{a,b}$.

\item
$\eta=-1, ab=-a-b;$ we then denote: $N_{a,b}=3C(-1)^{\times}_{a,b}$. 

\item
$\eta=\half, ab=\half a+\half b;$ we then denote: $N_{a,b}=J_{a,b}$.
\end{enumerate}
Furthermore,
\begin{itemize}
\item[(4)]
the algebras  $N_{a,b}$ in cases $(2)$ and $(3)$ above do not have an identity element.
\end{itemize}
\end{lemma}

The following proposition deals with $2$-generated $3$-dimensional subalgebras.

\begin{prop}[Proposition 4.6 \cite{hrs}]\label{prop 3dim}
Let $a,b\in\gD$ with $a\ne b$.  Then $N_{a,b}$ is $3$-dimensional
precisely when $ab\ne 0$ and there exists $0\ne \gs\in N_{a,b}$
and 
a scalar $\gvp=\gvp_{a,b}\in\ff$ such that if we set $\pi=\pi_{a,b}=(1-\eta)\gvp-\eta,$
then
\begin{enumerate}
\item
$ab=\gs+\eta a+\eta b;$

\item
$\gs v=\pi v,$ for all $v\in\{a,b,\gs\}.$
\end{enumerate}
furthermore
\begin{itemize}
\item[(3)]
$N_{a,b}$ contains an identity element if and only if $\pi\ne 0,$
in which 
case $1_{a,b}=\frac{1}{\pi}\gs$.
\end{itemize}
When $N_{a,b}$ is $3$-dimensional we denote: $N_{a,b}=B(\eta,\gvp)_{a,b},$ where $\gvp\in\ff$
is the scalar mentioned above.
\end{prop}

\text{From} now on we assume that {\bf $\gD_{\cala}$ is connected}.
 Note that by \cite[Lemma 6.4]{hss},
$\gD_{\cala}$ is connected iff $\gD$ is connected.
Further, we assume
that {\bf $a,b\in\gD$ are distinct with $\gt(a)=\gt(b)$}.

\begin{prop}[Proposition 6.5 in \cite{hss}]\label{prop ta=tb}
$ab=0$ and
\begin{enumerate}
\item
for any   
$c\in \gD\sminus\{a,b\}$ exactly one the following holds:
\begin{itemize}
\item[(i)]
$ac=bc=0.$

\item[(ii)]
$\eta=\half,$ and for some $x\in\{a,b\}=\{x,y\},$ we have $N_{x,c}=B(\half,0)_{x,c}$ is $3$-dimensional,
$N_{y,c}=J_{y, c}$  and $N_{y,c}\subset N_{x,c}$.
Further   $a+b=1_{x,c}$.

\item[(iii)]
$\eta=\half,$ 
$N_{a,c}=N_{b,c}$ is $3$-dimensional
and   $a+b=1_{a,c}$.  
\end{itemize}
\item
If $d$ is an $\eta$-axis in $A$ such that $\tau(d)=\tau(a),$
then $d\in\{a,b\}$.
\end{enumerate}
\end{prop}
\begin{proof}[Proof sketch]
By \cite[Lemma 3.2.1]{hss}, for any $c\in\gD,$ 
we have $ac=0\iff c^{\gt(a)}=c,$ and since, by definition, $a^{\gt(b)}=a^{\gt(a)}=a,$
we see that $ab=0$.
 
If $ac=0,$ then, as above  $bc=0$ (and vice versa),
so (i) holds.  Hence we may assume that $ac\ne 0\ne bc$.

If $\eta\ne\half,$ then by \cite[Proposition 6.5]{hrs}, and since $\gD$ is connected,
$a=b,$ a contradiction.  Thus $\eta=\half$.

Now consider  
\[
V:=N_{c,c^{\tau(a)}}\subseteq N_{a,c}\cap N_{b,c}.
\]
$V$ is either $2$ or $3$-dimensional.
If $V$ is $3$-dimensional, then $N_{a,c}=V=N_{b,c},$
and since $ab=0,$ one shows that $a+b=1_{a,c}$ (\cite[Lemma 3.2.5]{hss}),  
so (iii) holds.

So suppose $V$ is $2$-dimensional.  If both $N_{a,c}$ and $N_{b,c}$
are $2$-dimensional, then they both equal to $N_{a,b}=\ff a\oplus\ff b$.
But then $c=a$ or $b,$ a contradiction.

Therefore without loss $N_{a,c}$ is $3$-dimensional and $V$ 
is  $2$-dim\-en\-sion\-al.
If $V=N_{b,c}$ then (ii) holds:  Clearly $N_{b,c}\subset N_{a,c}$ and
$a+b=1_{a,c},$ and then a careful analysis of the situation gives
(ii).

The case where both $N_{a,c}$ and $N_{b,c}$ are $3$-dimensional
and $V$ is $2$-dimensional is the hardest case and some
precise work is required to get a contradiction.
\end{proof}

\begin{prop}\label{prop e}
$\eta=\half$ and 
\begin{enumerate}
\item
$xa\ne 0\ne xb,$ for all  $x\in\gD\sminus\{a,b\};$

\item
$A$ contains an identity element $\e=a+b;$

\item
for any $x\in\gD$ such that   
$N_{a,x}$ is $3$-dimensional we have $\e=1_{a,x}$.
\end{enumerate}
\end{prop}
\begin{proof} 
Let $d(\ ,\ )$
be the distance function on $\gD$.  Let
\[
\gD_1(a):=\{x\in\gD\mid d(a,x)=1\}.
\]
Since $\gD$ is connected $\gD_1(a)\ne\emptyset$.  Also,
by Proposition \ref{prop ta=tb}(1i), $\gD_1(a)=\gD_1(b)$.   
Let $c\in \gD_1(a)$. By Proposition \ref{prop ta=tb}, $\eta=\half$
and after perhaps interchanging $a$ and $b,$ $N_{a,c}$ is
$3$-dimensional and $a+b=1_{a,c}$. 
Set
\[
\e=1_{a,c}=a+b,
\]
then
\[
\e c=c,\text{ for all }c\in\gD_1(a).
\]

Let $y\in\gD\sminus\gD_1(a)$ be at distance $2$ from $a$ in $\gD,$
and let
\[
x\in\gD_1(a)\cap\gD_1(y).
\]
Without loss $N_{a,x}$ is $3$-dimensional and $\e=1_{a,x}$.
Now
\begin{itemize}
\item  
$ay=0=by\implies \e^{\gt(y)}=(a+b)^{\gt(y)}=a^{\gt(y)}+b^{\gt(y)}=a+b=\e.$

\item
$\e^{\gt(x)}=\e$ because $\e=1_{a,x}.$

\item
$\e y=0$ so $\e y^{\gt(x)}=0.$

\item
$\e x=x$ so $\e x^{\gt(y)}=x^{\gt(y)}.$

\item
$W:={\rm Span}\big(\{y,\ y^{\tau(x)}\}\big)\cap {\rm Span}\big(\{x,\ x^{\tau(y)}\}\big)\ne \{0\}$.
Indeed, $W$ 
is the intersection of two $2$-dimensional subspaces of $N_{x,y}$ which is of dimension
at most $3$.

\item
$\e$ both annihilates and acts as identity on $W,$ a contradiction.
\end{itemize}

Hence $\gD_1(a)=\gD\sminus\{a, b\}$ and clearly $d(a,b)=2$
in $\gD$.  But now, as we saw above, $\e c=c$ for
all $c\in\gD$.   
It follows that $\e$
is the identity of $A$ and (3) holds as well.
\end{proof}
 
We are now in a position to prove Theorem \ref{thm main}.
\medskip

\begin{proof}[{\bf Proof of Theorem \ref{thm main}}]
We show that the hypotheses of Theorem \ref{thm jvb} are satisfied.
By Proposition \ref{prop e}, $\eta=\half$ and $a+b=\e_A$. 
Let $c\in\gD$.  Then
%
\begin{gather*}\label{eq vv}
\textstyle{
v_av_c=(a-\half\e)(c-\half\e)=}\\\notag
\textstyle{ac-\half a-\half c+\frac{1}{4}\e=
\gs_{a,c}+\frac{1}{4}\e.}
\end{gather*}

 Clearly $v_av_c\in\ff\e$ if $c\in\{a,b\}$.
Otherwise, by Proposition \ref{prop e}(1), $ac\ne 0$.  
If $N_{a,c}$ is $2$-dimensional, then since $ac\ne 0,$
$\gs_{a,c}=0,$ and so $v_av_c\in \ff\e$.
If $N_{a,c}$ is $3$-dimensional, then by Proposition \ref{prop e}(3), $\e=1_{a,c}$.
Furthermore by [HRS],  $\gs_{a,c}=\pi_{a,c}1_{a,c}=\pi_{a,c}\e,$
for some $\pi_{a,c}\in\ff,$ and again $v_av_c\in\ff\e$.
\end{proof}
%
%
%
\section{$3$-transpositions and Matsuo algebras}\label{}
Recall that a set of axes $\cala$
	is closed iff $a^{\gt_b}\in\cala,$ for all $a, b\in\cala$.
In this section $A$ is a primitive axial algebra of Jordan type $\eta$
generated by a closed set of $\eta$-axes $\cala,$ such that $|\cala|>1$.

Let $G$ be a group generated by a normal set of involutions $D$.
Recall that $D$ is called {\it a set of $3$-transpositions in $G$}
if $|st|\in\{1,2,3\},$ for all $s,t\in D$.  The group
$G$ is then called a {\it $3$-transposition group}.

Let $D$ be a normal set of $3$-transpositions in the group $G$ that generate $G$. 
The {\bf Matsuo algebra} associated with the pair $(G,D),$ denoted here $M_{\gd}(G,D),$ is
defined as follows.  As a vector space over $\ff$ it  has the basis $D$.  Multiplication
is defined for $x,y\in D$ as follows
\begin{equation*}
x\cdot y=
\begin{cases}
x, &\ {\rm if}\   y=x\\
0, &\ {\rm if}\ |xy|=2\\
\gd(x+y-x^y), &\ {\rm if}\ |xy|=3.
\end{cases}
\end{equation*}
This is extended by linearity to the entire algebra. (Note that we denote multiplication in $G$ by
juxtaposition and in $M_{\gd}(G,D)$ by dot.) By \cite[Theorem 6.2]{hrs}, $M_{\gd}(G,D)$
is a primitive axial algebra of Jordan type $2\gd$.

The purpose of this section
is to prove the following Theorem:

\begin{thm}\label{thm matsuo}
Suppose that the graph $\gD_{\cala}$ is connected.
Let $D:=\{\gt_a\mid a\in \cala\}$ and $G=\lan D\ran$.
Assume that the map $a\mapsto\gt_a$ on $\cala$ is injective and that $D$ is a set of $3$-transpositions in $G$.
Then $A$ is a quotient of the Matsuo algebra $M_{\frac{\eta}{2}}(G,D)$.
\end{thm}

\begin{remark}
Theorem \ref{thm matsuo} was proved in \cite[Theorem 6.3]{hrs} for $\eta\ne\half$.
The proof for $\eta=\half$ needed a correction, in view of \cite{hss}.
Note that the summand $\oplus_{i\in I}\ff$ does not appear in
Theorem \ref{thm matsuo} since we are assuming that $\gD_{\cala}$ is connected.
We also mention that for $\eta\ne\half,$ the map on $\cala$ defined by $a\mapsto\gt_a$
is always injective, by \cite[Proposition 6.5]{hss}, and since $\gD_{\cala}$
is connected.

We included a proof of Theorem \ref{thm matsuo} for all $\eta$ for completeness.
\end{remark}

\begin{lemma}\label{lem ab}
$ab=\frac{\eta}{2}a+\gvp_{a,b} b-\frac{\eta}{2}a^{\gt_b},$ for all $a, b\in\cala$. 
\end{lemma}
\begin{proof}
Clearly this holds when $a=b,$ so assume $a\ne b$.
Suppose first that $N_{a,b}$ is $2$-dimensional.  We use \cite[Lemma 3.1.2]{hss}.
If $N_{a,b}=2B_{a,b},$ then $ab=0, \gvp_{a,b}=0,$ and $a^{\gt_b}=a$ (see also \cite[Lemma 3.2.1]{hss}),
so the claim holds.

Suppose next that $N_{a,b}=3C(-1)^{\times}_{a,b}$.  Then $\eta=-1, ab=-a-b, \gvp_{a,b}=-\half$
and $a^{\gt_b}=-a-b$ (see also \cite[Lemma 3.1.8]{hss}), so the claim holds.

Assume that $N_{a,b}=J_{a,b}$.  Then $\eta=\half, ab=\half a+\half b, \gvp_{a,b}=1$ and
$a^{\gt_b}=2b-a$ (see also \cite[Lemma 3.1.9]{hss}), so again the claim holds.
 
We may assume that $N_{a,b}$ is $3$-dimensional.  Set $\gvp:=\gvp_{a,b}$. By\linebreak \cite[Theorem 3.1.3(6)]{hss}, 
$a^{\tau(b)}=-\frac{2}{\eta}\gs-\frac{2(\eta-\gvp)}{\eta}b-a$.
Also, $\gs=ab-\eta a-\eta b$.
Hence we get
\begin{alignat*}{3}
\textstyle{\frac{2}{\eta}\gs} &\textstyle{=-a -\frac{2(\eta-\gvp)}{\eta}b-a^{\gt_b}}& &\iff\\
\gs&=\textstyle{-\frac{\eta}{2}a-(\eta-\gvp)b-\frac{\eta}{2}a^{\gt_b}} & &\iff\\
ab&=\textstyle{\frac{\eta}{2}a+\gvp b-\frac{\eta}{2}a^{\gt_b}.}\qedhere
\end{alignat*}
\end{proof}

\begin{cor}[See Corollary 1.2 in \cite{hrs}]\label{cor gen}
$A$ is spanned over $\ff$ by $\cala$.
\end{cor}
\begin{proof}
This is immediate from Lemma \ref{lem ab} and the definition of a closed set of axes.
\end{proof}

\begin{lemma}\label{lem matsuo}
Suppose that 
\[\tag{$*$}
\text{the map $a\mapsto \gt_a$
on $\cala$ is injective.}
\]
Let $a, b\in\cala$ be distinct.  Then
\begin{enumerate}
\item
if $(\gt_a\gt_b)^2=1,$ then $ab=0$.

\item
if $(\gt_a\gt_b)^3=1,$ then $\gvp_{a,b}=\frac{\eta}{2}$.
\end{enumerate}
\end{lemma}
\begin{proof}
(1):\quad By \cite[Lemmas 3.2.7(2) and 3.1.6(2)]{hss} and by $(*),$
$N_{a,b}=2B_{a,b},$ so (1) holds (see also \cite[Lemma 3.1.2(1a)]{hss}).
\medskip

\noindent
(2):\quad If $\eta\ne\half,$ then (2) follows from \cite[Proposition 4.8]{hrs}.
So suppose $\eta=\half$.  By \cite[Lemma 3.2.7(1) and Corollary 3.3.2]{hss} 
and by $(*),$ we get $\gvp_{a,b}=\frac{1}{4}.$
\end{proof}

We can now prove Theorem \ref{thm matsuo}.
 
\begin{proof}[Proof of Theorem \ref{thm matsuo}]\hfill
\medskip

Set $M:=M_{\frac{\eta}{2}}(G,D)$.  We claim that the map 
\[
f\colon M\to A:\ \gt_a\mapsto a,
\]
extended by linearity is a  surjective algebra homomorphism. 
Note that $f$ is well defined
since the map $a\mapsto\gt_a$ is injective on $\cala$.
 
Now $f$ is surjective by Corollary \ref{cor gen}.
Next we need to check
that 
\[\tag{$*$}
f(\gt_a\cdot\gt_b)=ab,\text{ for all }a,b\in\cala.
\]  
If $a=b,$ then
$\gt_a\cdot\gt_b=\gt_a$, and $ab=a,$ so $(*)$ holds.

If $|\gt_a\gt_b|=2,$ then $\gt_a\cdot\gt_b=0,$ while by Lemma \ref{lem matsuo}(1),
$ab=0,$ so $(*)$ holds in this case as well.

Finally assume that $|\gt_a\gt_b|=3$.  Then 
\[\textstyle{
\gt_a\cdot\gt_b=\frac{\eta}{2}(\gt_a+\gt_b-\gt_a^{\gt_b})=\frac{\eta}{2}(\gt_a+\gt_b-\gt_{_{a^{\gt_b}}}),}
\]
where the last equality follows from the standard fact that $\gt_a^{\gt_b}=\gt_{_{a^{\gt_b}}}$.
Thus $f(\gt_a\cdot\gt_b)=\frac{\eta}{2}(a+b-a^{\gt_b})$.  However, by Lemma \ref{lem matsuo}(2) and Lemma \ref{lem ab},
$ab=\frac{\eta}{2}(a+b-a^{\gt_b}),$ so $(*)$ holds in this case as well and the proof of the
theorem is complete.
\end{proof}

\section{The existence of a Frobenius form}

Recall that a non-zero bilinear form $(\cdot\, ,\, \cdot)$ on an algebra $A$ 
is called {\bf Frobenius} if the form associates with the algebra product, 
that is, 
\[(ab,c)=(a,bc)\]
for all $a,b,c\in A$.

For primitive axial algebras of Jordan type $\eta,$ we specialize the 
concept of Frobenius form further by asking that the condition 
$(a,a)=1$ be satisfied for each $\eta$-axis $a$.

The purpose of this section is to prove the following theorem:

\begin{thm}\label{thm frob}
Let $A$ be a   primitive axial algebra of Jordan type $\eta$.
Then $A$ admits a Frobenius form.
\end{thm}

The proof of Theorem \ref{thm frob} depends on two properties of primitive
axial algebras of Jordan type.  The first is 
Corollary \ref{cor gen}.
The second is  proven 
in \cite{hrs} (Lemma \ref{symmetric} below).

For an $\eta$-axis $a\in A$, let $\gvp_a$ be the projection function 
with respect to $a$. That is, for $u\in A$, we have that 
$u=\gvp_a(u)a+u_0+u_\eta$, where $u_0$ and $u_\eta$ are eigenvectors 
of the adjoint linear transformation ${\rm ad}_a$ for the eigenvalues 
$0$ and $\eta$, respectively.

\begin{lemma}[Lemma 4.4 in \cite{hrs}]\label{symmetric}
For a primitive axial algebra $A$ of Jordan type and for any $\eta$-axes $a,b\in A$, we have 
$\gvp_a(b)=\gvp_b(a)$.
\end{lemma}

Note that the constant $\gvp_{a,b},$ that we used earlier for $\eta$-axes $a,b,$ is the same as $\gvp_a(b)$.
\clearpage

\begin{proof}[Proof of Theorem \ref{thm frob}]\hfill
\medskip

We start by defining the bilinear form $(\cdot\, ,\, \cdot)$ on $A$.
Using Corollary \ref{cor gen} 
we can select a basis $\calb$ of $A$ consisting 
of $\eta$-axes, and we let
\[
(a\, ,\, b)=\gvp_a(b),\text{ for all } a, b\in\calb.
\]
Extending by linearity we get the bilinear form $(\cdot\, ,\, \cdot)$.
Note that 
Lemma \ref{symmetric} implies that 
$(\cdot\, ,\, \cdot)$ is symmetric. 

\begin{lemma}
\begin{enumerate}
\item
$(a\, ,\, u)=\gvp_a(u),$ for all $\eta$-axes $a\in A$ and all $u\in A;$

\item
$(a\, ,\, a)=1,$ for all $\eta$-axes $a\in A;$

\item
$(\cdot\, ,\, \cdot)$ is invariant under automorphisms of $A$.
\end{enumerate}
\end{lemma}
\begin{proof}
(1\&2):\quad 
Let $a$ be an $\eta$-axis and suppose that 
\[\tag{$*$}
\gvp_a(b)=(a,b), \text{ for all }b\in\calb.
\] 
Since $\gvp_a$ is linear, 
\begin{gather*}
\gvp_a(u)=
\gvp_a(\sum_{b\in\calb}\ga_b b)=\sum_{b\in\calb}\ga_b\gvp_a(b)\\
=\sum_{b\in\calb}\ga_b(a,b)
=(a,\sum_{b\in\calb}\ga_b b)=(a,u),
\end{gather*} 
and (1) holds for $a$.  Now if $a\in\calb,$ then $(*)$ holds by definition, so (1) holds for $a$.
Suppose $a\notin\calb$.  Let 
$b \in \calb$.  
Then $\gvp_a(b)=\gvp_b(a),$ by Lemma \ref{symmetric},
and $\gvp_b(a)=(b,a),$ as (1) holds for $b$.  Finally, since $(\cdot\, ,\, \cdot)$
is symmetric $(b,a)=(a,b),$ so $\gvp_a(b)=(a,b),$ and $(*)$ holds for any $\eta$-axis
$a$.  This shows that (1) holds. 

In particular, for every $\eta$-axis $a\in A$, we 
have that $(a,a)=1$, since, clearly, 
$\gvp_a(a)=1$. Thus (2) holds. 
\medskip

\noindent
(3):\quad  
Let $\psi\in\Aut(A)$, if $u=\gvp_a(u)a+u_0+u_\eta$ is the decomposition 
of $u\in A$ with respect to the $\eta$-axis $a,$ then 
$u^\psi=\gvp_a(u)a^\psi+u_0^\psi+u_\eta^\psi$ is the decomposition 
of $u^\psi$ with respect to the $\eta$-axis $a^\psi$. Hence 
$\gvp_{a^\psi}(u^\psi)=\gvp_a(u)$, and so $(a^\psi,u^\psi)=(a,u)$. 
Finally, taking an arbitrary $v\in A$ and decomposing it with respect 
to the basis $\calb$ as $v=\sum_{b\in\calb}\ga_b b$, we get that 
$(v^\psi,u^\psi)=(\sum_{b\in\calb}\ga_b b^\psi,u^\psi)=
\sum_{b\in\calb}\ga_b(b^\psi,u^\psi)=\sum_{b\in\calb}\ga_b(b,u)=
(\sum_{b\in\calb}\ga_b b,u)=(v,u)$. So indeed, $(\cdot\, ,\, \cdot)$ is 
invariant under the automorphisms of $A$. 
\end{proof}

\begin{lemma}
For every $\eta$-axis $a\in A$, 
different eigenspaces of ${\rm ad}_a$ are orthogonal with respect to 
$(\cdot\, ,\, \cdot)$. 
\end{lemma}
\begin{proof}
Clearly, if 
$u\in A_0(a)+A_\eta(a)$ then $(a,u)=\gvp_a(u)=0$. Hence $A_1(a)=\ff a$ is 
orthogonal to both $A_0(a)$ and $A_\eta(a)$. It remains to show that 
these two are also orthogonal to each other. Let $u\in A_0(a)$ 
and $v\in A_\eta(a)$, the fact that $(\cdot\, ,\, \cdot)$ is invariant under
$\tau_a$ gives us $(u,v)=(u^{\tau_a},v^{\tau_a})=(u,-v)=-(u,v)$. 
Clearly, this means that $(u,v)=0$.
\end{proof}
 
We are now ready to complete the proof that $(\cdot\, ,\, \cdot)$ associates 
with the algebra product. Note that the identity 
\[
(a,bc)=(ab,c)
\]
that we need to prove is linear in $a$, $b$, and $c$. In particular, 
since $A$ is spanned by $\eta$-axes, we may assume that $b$ is an $\eta$-axis. 
Furthermore, since $A$ decomposes as the sum of the eigenspaces of 
${\rm ad}_b$, we may assume that $a$ and $c$ are eigenvectors of ${\rm ad}_b$, 
say, for the eigenvalues $\mu$ and $\nu$. We have two cases:

\noindent
If  $\mu=\nu$ then 
\[
(a,bc)=(a,\nu c)=\nu(a,c)=\mu(a,c)=(\mu a,c)=(ba,c)=
(ab,c).
\] 
If $\mu\neq\nu$ then 
\[
(a,bc)=\nu(a,c)=0=\mu(a,c)=(ab,c),
\] 
since $A_\mu(b)$ and $A_\nu(b)$ are orthogonal to each other. 
Thus, in both cases we have the desired equality $(a,bc)=(ab,c)$, 
proving that the form $(\cdot\, ,\, \cdot)$ is Frobenius.
\end{proof}


\end{document}